\documentclass[a4paper, 12pt]{amsart}

\usepackage{amssymb}
\usepackage{amsthm}
\usepackage{amsmath}
\usepackage[mathscr]{eucal}
\usepackage{comment}

\usepackage{stmaryrd}

\usepackage{manfnt}
\usepackage{enumitem}

\AtBeginDocument{%
   \def\MR#1{}
}

\theoremstyle{plain}
\newtheorem{thm}{Theorem}[section]		
\newtheorem{prop}[thm]{Proposition}
\newtheorem{cor}[thm]{Corollary}
\newtheorem{lem}[thm]{Lemma}

\theoremstyle{definition}
\newtheorem{df}{Definition}[section]

\newtheorem{ques}[thm]{Question}

\theoremstyle{remark}
\newtheorem{rmk}{Remark}[section]
\newtheorem*{ac}{Acknowledgements}

\newcommand{\zz}{\mathbb{Z}}
\newcommand{\qq}{\mathbb{Q}}
\newcommand{\rr}{\mathbb{R}}

\DeclareMathOperator{\card}{Card}

\DeclareMathOperator{\met}{Met}

\newcommand{\ult}[2]{\mathrm{UMet}(#1; #2)}

\DeclareMathOperator{\umetdis}{\mathcal{UD}}

\DeclareMathOperator{\yodiam}{diam}

\newcommand{\yodisdis}{M}

\newcommand{\yochara}{\mathscr}

\newcommand{\youfin}{\yochara{N}}

\newcommand{\yofin}{\yochara{F}}

\newcommand{\yosub}{\subseteq}

\newcommand{\yoimage}[1]{\mathrm{Im}(#1)}

\newcommand{\yorubdim}[1]{\overline{\dim}_{\mathrm{B}}(#1)}
\newcommand{\yorhdim}[1]{\dim_{\mathrm{H}}(#1)}

\newcommand{\yorpdim}[1]{\dim_{\mathrm{P}}(#1)}
\newcommand{\yoradim}[1]{\dim_{\mathrm{A}}(#1)}

\newcommand{\yordimsetu}[3]{\mathrm{UFD}(#1; #2; #3)}

\DeclareMathOperator{\ubdim}{\overline{\dim}_{B}}
\DeclareMathOperator{\hdim}{\dim_{H}}

\DeclareMathOperator{\pdim}{\dim_{P}}
\DeclareMathOperator{\adim}{\dim_{A}}

\newcommand{\yorlset}{\mathcal{A}}

\newcommand{\yorzeron}{\yorbd{n}_{0}}

\newcommand{\yorbd}[1]{\mathbf{#1}}

\newcommand{\yonbsp}[1]{N(#1)}

\newcommand{\yocantorc}{\Gamma}

\newcommand{\yocmet}[2]{\mathrm{Cpu}(#1, #2)}

\newcommand{\yomapsp}[2]{\mathrm{C}(#1, #2)}
\newcommand{\yomaps}[2]{\mathrm{C}_{0}(#1, #2)}

\newcommand{\yomapsb}[2]{\mathrm{BMap}(#1, #2)}
\newcommand{\yobmet}{\mathcal{SM}}

\newcommand{\yomaindis}{\triangledown}

\newcommand{\yosph}{\mathbb{S}}

\newcommand{\yoweight}[1]{w(#1)}

\newcommand{\yoseqsp}[1]{\mathrm{Seq}(#1)}

\newcommand{\yoquin}{\mathfrak{A}}
\newcommand{\yoopsp}{\mathbb{O}}

\newcommand{\yosind}{\mathrm{ind}}
\newcommand{\yolind}{\mathrm{Ind}}
\newcommand{\yocdim}{\mathrm{dim}}

\newcommand{\yoclosym}{\mathfrak{c}}
\newcommand{\yocutsp}[3]{[#1]_{\yoclosym(#2, #3)}}
\newcommand{\yocutmem}[3]{[#1]_{\yoclosym(#2, #3)}}
\newcommand{\yocutdis}[2]{[#1]_{\yoclosym(#2)}}

\makeatletter
\@addtoreset{equation}{section}

\makeatother

\begin{document}

\title[Constructions]
{
Constructions of 
Urysohn universal ultrametric
 spaces
}
\author[Yoshito Ishiki]
{Yoshito Ishiki}

\address[Yoshito Ishiki]
{\endgraf
Photonics Control Technology Team
\endgraf
RIKEN Center for Advanced Photonics
\endgraf
2-1 Hirasawa, Wako, Saitama 351-0198, Japan}

\email{yoshito.ishiki@riken.jp}

\date{\today}

\begin{abstract}
In this paper, 
we give 
 new constructions of Urysohn universal ultrametric spaces. 
We first characterize 
a Urysohn universal ultrametric subspace 
 of the space of all 
  continuous functions whose images contain the zero, 
   from 
a zero-dimensional compact Hausdorff space 
without isolated points into the space of non-negative  real numbers equipped with the nearly discrete topology. 
As a consequence, 
the whole function  space is Urysohn universal, 
which  can be considered as 
a 
 non-Archimedean analog of 
 Banach--Mazur theorem. 
As a more application, we prove that 
the space of all continuous pseudo-ultrametrics on a 
zero-dimensional compact Hausdorff space
with an accumulation point  is a 
Urysohn universal ultrametric space. 
This result  can be considered as a variant  of 
 Wan's construction of Urysohn universal ultrametric space via the Gromov--Hausdorff ultrametric space. 
\end{abstract}

\maketitle
\section{Introduction}\label{sec:intro}
For a   class $\yochara{C}$
of metric spaces, 
a metric space $(X, d)$ is said to be 
\emph{$\yochara{C}$-injective} if 
for all $(A, a)$ and $(B, b)$ in $\yochara{C}$ and 
for all isometric embeddings 
$\phi\colon (A, a)\to (B, b)$ and 
$\psi\colon (A, a)\to (X, d)$, 
there exists an isometric embedding 
$\theta\colon (B, b)\to (X, d)$ such that 
$\theta\circ \phi=\psi$. 
This definition roughly coincides with  the notion of 
injective objects in category theory. 
We denote by $\yofin$
the class of all finite metric spaces. 
Urysohn \cite{Ury1927} constructed 
a complete separable $\yofin$-injective 
metric space $\mathbb{U}$, 
which is 
 nowadays called  
the \emph{Urysohn universal 
metric space}. 
In this paper, we investigate an
ultrametric analogue of $\mathbb{U}$. 
For more information of $\mathbb{U}$ and related spaces, 
we refer the readers to, for instance,  \cite{MR2667917},
\cite{MR2435145}, 
\cite{MR952617}, 
\cite{MR2435148}, 
\cite{MR2051102}, 
\cite{MR1900705}
and 
\cite{MR2277969}. 
In the recent years, 
the theory of Urysohn type metric spaces has 
applications to logic and  model theory 
(see for example, \cite{MR3583613}
and \cite{MR3413493}). 
Indeed, the space $\mathbb{U}$ can be 
regarded as a generalization of
 the random graph in model theory
(see also  \cite[Subsection 3.3]{MR2667917}).

A metric $d$ on a set $X$ is said to be 
an \emph{ultrametric} if 
for all $x, y, z\in X$, it satisfies the 
\emph{strong triangle inequality}
$d(x, y)\le d(x, z)\lor d(z, y)$, 
where $\lor$ stands for the maximum operator on $\rr$. 
If a pseudo-metric $d$ satisfies the strong triangle inequality, 
then $d$ is called a \emph{pseudo-ultrametric}.

A set $R$ is said to be a \emph{range set} if 
it is a subset of $[0, \infty)$ 
and $0\in R$. 
If $R$ is a range set, and  an ultrametric $d$ (resp.~pseudo-ultrametric) on a 
set $X$ satisfies $d(x, y)\in R$ for all $x, y\in R$, 
then we 
call $d$ an  \emph{$R$-ultrametric} or 
\emph{$R$-valued ultrametrics} 
(resp.~\emph{$R$-pseudo-ultrametric} or 
\emph{$R$-valued pseudo-ultrametrics}).

For a range set $R$, 
we denote by $\youfin(R)$ the 
class of all 
finite $R$-valued 
ultrametric spaces. 
The main subject of this paper is 
to provide
new
constructions of 
 complete $\youfin(R)$-injective $R$-ultrametric spaces for every  range set $R$.

In the title and the abstract of this paper, 
 the term ``the Urysohn universal ultrametric spaces''
  means  injective ultrametric spaces. 
However, 
from now on, 
for a range set $R$, 
we use the term ``the $R$-Urysohn universal ultrametric spaces''
 as 
 the  separable 
$\youfin(R)$-injective complete 
$R$-ultrametric space. 
Note that such an space is unique up to isometry
and 
note that if $R$ is uncountable, 
a $\youfin(R)$-injective $R$-ultrametric space is non-separable
(see for example, 
\cite[Section 2]{MR2754373} and 
\cite[(12) in Theorem 1.6]{MR3782290}).

There are several construction of 
$\youfin(R)$-injective $R$-valued 
ultrametric spaces. 
Similarly to the ordinary Urysohn universal metric space
(see \cite{Ury1927} and \cite{MR952617}), 
$\youfin(R)$-injective $R$-ultrametric spaces
can be obtained by 
the Urysohn amalgamation method  
(the Fra\"{i}ss\'{e} limit)
(see \cite{MR1843595}), 
and 
the way using Kat\'{e}tov function spaces
(see \cite{MR2754373}). 
It is also  known that we can construct  $\youfin(R)$-injective $R$-ultrametric spaces as the spaces of  
branches (rays) of trees
(see \cite{MR2754373} and  \cite{MR2667917})). 
Wan \cite{MR4282005} 
proved
the $\youfin(R)$-injectivity
of 
the non-Archimedean 
Gromov--Hausdorff space; namely, 
the space of all isometry classes of  compact $R$-valued 
ultrametric spaces equipped with the Gromov--Hausdorff.

Before explaining our constructions, 
we prepare some concepts. 
For a range set $R$, 
we 
define an ultrametric 
$\yodisdis_{R}$ on $R$ by 
\[
\yodisdis_{R}(x, y)=
\begin{cases}
x\lor y & \text{if $x\neq y$}\\
0  & \text{if $x=y$}. 
\end{cases}
\]
Then the metric 
$\yodisdis_{R}$ is an ultrametric on $R$. 
This construction  was given by  
Delhomm\'{e}--Laflamme--Pouzet--Sauer \cite[Proposition 2]{MR2435142}, 
and it also  can be found in  \cite{MR2854677}, 
\cite{Ishiki2021ultra}, and \cite{Ishiki2022highpower}. 
Based on  the notion of \cite{MR350705}, 
we call  $\yodisdis_{R}$ (resp.~the topology generated by $\yodisdis_{R}$)
 the  \emph{nearly discrete (ultra)metric on $R$}
(resp.~the \emph{nearly discrete topology}). 
The space $(R, \yodisdis_{R})$ is 
as significant  for ultrametric spaces as 
the space $[0, \infty)$ or $\rr$ with the  Euclidean topology
in the theory of 
usual metric spaces. 
There are 
 two key points of  this paper. 
 The first is the recognition of the importance of  $(R, \yodisdis_{R})$. 
 The second is the discovery of  the tenuous sets
 as the compact subspaces of 
$(R, \yodisdis_{R})$ (for the definition, 
see Subsection \ref{subsec:tenuous}).

For a 
 topological space $X$ and
a range set  $R$, we 
denote by $\yomapsp{X}{R}$
the set of all continuous maps from 
$X$ to the nearly discrete space $R$. 
We also denote by
$\yomaps{X}{R}$ the set of all 
$f\in \yomapsp{X}{R}$ such that $0\in f(X)$. 
For $f, g\in \yomapsp{X}{R}$, 
we define $\yomaindis(f, g)$ by 
the infimum of all 
$\epsilon \in (0, \infty)$ such that 
$f(x)\le g(x)\lor \epsilon$ and 
$g(x)\le f(x)\lor \epsilon$ for all 
$x\in X$. 
By abuse of notations, we use the same symbol $\yomaindis$
as the restricted function 
$\yomaindis|_{\yomaps{X}{R}^{2}}$. 
Then the function 
$\yomaindis$ is 
an ultrametric on 
$\yomaps{X}{R}$ taking values in $[0, \infty]$. 
In this paper, 
since we only consider the case where $X$ is compact, 
the function $\yomaindis$ is actually an ultrametric in 
the ordinary sense.

In Theorem \ref{thm:favoid} in  this paper, 
for a subset $E$ of $\yomaps{X}{R}$, 
we shall  give a sufficient and 
necessary condition for the 
$\youfin(R)$-injectivity of 
$E$. 
As a consequence, we conclude that 
the space $(\yomaps{X}{R}, \yomaindis)$ itself  is 
$\youfin(R)$-injective
(see Theorem \ref{thm:stand}), 
which is our first construction. 
Since all separable $R$-valued ultrametric spaces 
are isometrically embeddable into 
a complete $\youfin(R)$-injective $R$-ultrametric 
space (see for example, \cite[Proposition 2.7]{MR2754373}), 
Theorem \ref{thm:stand} can be regarded 
 as 
an analog of the Banach--Mazur theorem 
stating that all separable metric spaces can be 
isometrically embedded into 
the space of all  functions on 
$[0, 1]$ equipped with the supremum metric. 
We can also consider 
the construction of 
$(\yomaps{X}{R}, \yomaindis)$ 
is a 
generalization of 
Vestfrid's construction 
in \cite{MR1354831}, 
which explains that 
the space of all decreasing sequences  in $[0, \infty)$ 
covergent to $0$
is $\youfin(R)$-injective.

For a topological space $X$, 
and a range set $R$, 
we denote by 
$\yocmet{X}{R}$ the 
set of all 
continuous $R$-valued  pseudo-ultrametrics $d\colon X\times X\to R$ on $X$, where $X\times X$ and $R$ are equipped with the 
product topology and the Euclidean topology, respectively (see Proposition \ref{prop:spectral}). 
For $d, e\in \yocmet{X}{R}$, 
we define $\umetdis_{X}^{R}(d, e)$ the 
infimum of all $\epsilon \in R$ such that 
$d(x, y)\le e(x, y)\lor \epsilon$ and 
$e(x, y)\le d(x, y)\lor \epsilon$ for all $x, y\in X$. 
The metric $\umetdis_{X}^{R}$ is 
a restricted metric of $\yomaindis$ and 
is an extension of ultrametrics on spaces of 
ultrametrics defined in the author's papers 
\cite{Ishiki2021ultra} and \cite{Ishiki2022highpower}. 
As in the case of $\yomaps{X}{R}$, 
we focus only on the case where $X$ is compact, 
the value $\umetdis_{X}^{R}(d, e)$ is 
always finite.

As an application of  Theorem \ref{thm:favoid}, 
we will prove  Theorem \ref{thm:contimetuniv}
asserting  that 
if $X$ is a compact Hausdorff space 
possessing an accumulation point, and 
$R$ is a range set,  
then the space 
$(\yocmet{X}{R}, \umetdis_{X}^{R})$ is 
$\youfin(R)$-injective and complete. 
This is our second construction and 
it  can be considered as a variant of 
 Wan's construction of Urysohn universal ultrametric space via the Gromov--Hausdorff ultrametric space
 \cite{MR4282005}. 
We also show that several subsets of
$\yocmet{X}{R}$ are 
$\youfin(R)$-injective
(see Theorem \ref{thm:manyinj}). 
For example, 
the set of all doubling ultrametrics is 
injective. 
To obtain this result, 
we use the author's results on  dense
subsets of spaces of ultrametrics
in 
\cite{Ishiki2021ultra} and
\cite{Ishiki2021dense}, 
and use the extension theorem of ultrametrics in 
\cite{Ishiki2022factor}.

The organization of this paper is as follows: 
Section \ref{sec:pre}
presents 
 some preliminaries. 
We introduce tenuous subsets of $[0, \infty)$, which plays an important role in this paper. 
We also investigate basic properties of 
$(\yomaps{X}{R}, \yomaindis)$, 
$(\yocmet{X}{R}, \umetdis_{X}^{R})$, 
and $\youfin(R)$-injective 
ultrametrics spaces. 
In Section \ref{sec:funcsp}, 
we explain our first construction 
(see Theorem \ref{thm:stand}). 
We also give a sufficient and necessary condition 
for the property that $X$ has no isolated points 
using the $\youfin(R)$-injectivity of 
$\yomaps{X}{R}$. 
Section \ref{sec:ultrametrics} is intended to 
 show that if 
$X$ is a compact Hausdorff space  possessing 
an accumulation point, then  $(\yocmet{X}{R}, \umetdis_{X}^{R})$ is 
$\youfin(R)$-injective
(see Theorem \ref{thm:contimetuniv}). 
This is our second construction. 
In Section \ref{sec:add}, 
we determine the topological type  of 
$\yomaps{X}{R}$ or $\yocmet{X}{R}$.

\section{Preliminaries}\label{sec:pre}
In this section, 
we prepare some basic  statements on ultrametrics, 
spaces of ultrametrics, and injective metric spaces.

\subsection{Generalities}\label{subsec:gen}
For a set $E$, 
the symbol  $\card(E)$ stands for the 
cardinality of $E$. 
For a metric space 
$(X, d)$, $a\in X$,
and $r\in (0, \infty)$, 
we denote by 
$B(a, r; d)$ (resp.~$U(a, r; d)$)
the closed ball (resp.~open ball) centered at $a$ 
with radius $r$. 
We also define $\yosph(a, r; d)
=\{\, x\in X\mid d(x, a)=r\, \}$. 
We often simply represent them as   
$B(a, r)$, $U(a, r)$,and $\yosph(a, r)$. 
In this paper, 
for a metric space $(X, d)$, and a subset $A$ of $X$, 
we sometimes represent   the restricted metric 
$d|_{A^{2}}$ as the same symbol to the ambient metric $d$
when no confusions can arise.

\subsubsection{Dimension}\label{subsec:gen}
A subset of a topological space is 
\emph{clopen} if it is closed and open in the 
ambient space. 
We define the zero-dimensionality of topological spaces in 
three ways. 
Let $X$ be a non-empty topological space. 
We write 
$\yosind(X)=0$ if 
$X$ has an open base consisting of clopen subsets of $X$. 
We also write 
$\yolind(X)=0$ if 
for all two disjoint closed subsets $A$ and $B$ of $X$, 
there exists a clopen subset $C$ such that 
$A\yosub C$ and $B\cap C=\emptyset$. 
The notation   $\yocdim(X)=0$ means that 
every finite open covering of $X$ has a 
refinement covering of $X$ consisting of 
finitely many mutually disjoint open subsets of $X$. 
The dimensions
$\yosind(X)$, 
$\yolind(X)$, 
and 
$\yocdim(X)$ 
are called the \emph{small inductive dimension}, 
the \emph{large inductive dimension},
and 
the \emph{covering dimension}, 
respectively. 
Of cause, we can define higher dimensional spaces, 
however, we omit it since we focus only on 
the $0$-dimensionality in this paper. 
In general,  these three dimensions can 
different from each other. 
However, as stated in  the 
next proposition, 
they are all  equal to $0$ 
if $X$ is compact Hausdorff 
and any  one of the three is $0$. 
The proof is  deduced from 
\cite[Corollary 2.2, Proposition 2.3, pp156--157]{MR0394604}. 
\begin{prop}\label{prop:zerodim}
Let $X$ be a non-empty compact Hausdorff space. 
If any one of the three $\yosind(X)$,  $\yolind(X)$, and $\yocdim(X)$ is equal to $0$, 
then 
$\yosind(X)=\yolind(X)=\yocdim(X)=0$ . 
\end{prop}

Based on  Proposition \ref{prop:zerodim}, 
in what follows, we will  say that a compact Hausdorff space $X$ is \emph{$0$-dimensional}
if any one of the three dimensions is $0$ (in this case, all the three dimensions are $0$). 

\subsubsection{Topological weights}

For a topological space $X$, 
we  denote by 
$\yoweight{X}$ the topological weight of $X$. 
Namely, $\yoweight{X}$ is equal to the 
minimum of cardinals of all open bases of $X$
(remark that some authors call  the cardinal 
$\max\{\yoweight{X}, \aleph_{0}\}$  the 
topological weight). 
Hence a  metrizable space $X$ is separable 
if and only if $\yoweight{X}\le \aleph_{0}$. 
Using compactness, we obtain the 
following lemma:

\begin{lem}\label{lem:22-22}
Let $X$ be a $0$-dimensional 
compact Hausdorff space
such that $\yoweight{X}$ is infinite. 
If $\mathcal{C}$ is the 
set of all clopen subsets of $X$, 
then 
$\card(\mathcal{C})
=\yoweight{X}$. 
\end{lem}

For a metric space
$(X, d)$, 
we say
a subset $A$ of $X$ is 
\emph{$r$-separated}
if 
we have $r\le d(x, y)$
 for all distinct $x, y\in A$.
The following  lemma can be proven 
by the definition of the topological weights. 
\begin{lem}\label{lem:leweight}
Let $(X, d)$ be a metric space. 
If a subset $E$ of $X$ is  $r$-separated,  
then we have $\card(E)\le \yoweight{X}$. 
\end{lem}

The next lemma can be proven 
by Lemma \ref{lem:leweight} with  taking  a maximal $(2^{-n})$-separated set from a dense subset for each $n\in \zz_{\ge 0}$, 
or  by  the fact that  every metric space has a 
$\sigma$-discrete open base. 
\begin{lem}\label{lem:weightle}
Let $(X, d)$ be a metric space. 
If $F$ is a dense subset of $X$, 
then $\yoweight{X}\le \card(F)$. 
\end{lem}

Remark  that Lemmas \ref{lem:leweight} and 
\ref{lem:weightle} are consequences of 
the fact that the weight of a metric space is 
equal to the minimum of cardinals of all 
dense subsets of the space.

\subsubsection{Properties of ultrametrics}
For a range set $R$, and 
for a topological space $X$, 
we denote by the 
$\ult{X}{R}$ the set of all 
$R$-valued ultrametrics that generate the 
same topology of $X$. 
Remark that 
$\ult{X}{R}\yosub 
\yocmet{X}{R}$. 
A topological space $X$ is said to be 
\emph{$R$-valued ultrametrizable} or 
\emph{$R$-ultrametrizable}
if 
$\ult{X}{R}\neq \emptyset$. 
We simply say that $X$ is \emph{ultrametrizable} if 
it is $[0, \infty)$-valued ultrametrizable. 
A range set $R$ is said to be 
\emph{characteristic} if 
for all $t\in [0, \infty)$, 
there exists $r\in R\setminus \{0\}$ such that 
$r\le t$. 
The next lemma gives a characterization of ultrametrizablity, 
and indicates the importance of  characteristic range
sets.  The proof follows from 
\cite[Proposition 2.14]{Ishiki2021ultra} and 
 \cite[Theorem I]{MR80905}
(note that
the property that $R$ is characteristic is equivalent to
saying  that 
 a range set $R$ has the countable coinitiality in the sense of \cite{Ishiki2021ultra}). 
\begin{prop}\label{prop:equivultmet}
Let $X$ be a topological space, 
and $R$ be a characteristic range set. 
Then the following are equivalent to each other:
\begin{enumerate}
\item 
The space
$X$ is metrizable and $\yolind(X)=0$; 
\item 
The space 
$X$ is ultrametrizable;
\item 
The space 
$X$ is $R$-ultrametrizable. 
\end{enumerate}
\end{prop}

The proofs of next three lemmas 
 are presented in 
 Propositions 
 18.2, 
 18.4, 
 and 18.5 in 
\cite{MR2444734}, 
respectively. 
\begin{lem}\label{lem:isosceles}
Let $X$ be a set, and 
$d$ be a pseudo-metric on $X$. 
Then $d$ satisfies the strong 
triangle inequality if and only if 
for all $x, y, z\in X$, 
the inequality $d(x, z)<d(z, y)$ implies 
$d(z, y)=d(x, y)$. 
\end{lem}

\begin{lem}\label{lem:ultraopcl}
Let $(X, d)$ be a pseudo-ultrametric space, 
$a\in X$ and $r\in [0, \infty)$. 
Then, the following statements are ture:
\begin{enumerate}
\item
The sets
$B(p, r)$ and $U(a, r)$ are 
clopen in $X$. 
\item 
For all $q\in B(p, r)$ (resp.~$q\in U(p, r)$), 
we have $B(p, r)=B(q, r)$ 
(resp.~$U(p, r)=U(q, r)$). 
\end{enumerate}

\end{lem}

\begin{lem}\label{lem:balldis}
Let $(X, d)$ be an ultrametric space, 
$r\in (0, \infty)$, and $a, b\in X$. 
Then for all $x\in B(a, r)$ and 
$y\in B(b, r)$, 
we have $d(x, y)=d(a, b)$. 
\end{lem}

\subsubsection{Nearly discrete spaces}
In this subsection, 
we verify  basic  properties of 
the nearly discrete space 
$(R, \yodisdis_{R})$. 

\begin{lem}\label{lem:nearlydis}
Let $R$ be a range set. 
Then the next statements are true. 
\begin{enumerate}[label=\textup{(\arabic*)}]
\item\label{item:nd:1}
The set $R\setminus \{0\}$ is 
discrete with respect to the nearly discrete space, i.e., 
all singletons are open. 
\item\label{item:nd:2}
If $R$ is characteristic, 
then $0$ is a unique accumulation point. 
\end{enumerate}
\end{lem}
\begin{proof}
The lemma follows from 
the fact that $U(x, x; \yodisdis_{R})=\{x\}$ and 
$U(0, x; \yodisdis_{R})=[0, x)\cap R$
for all 
$x\in R\setminus \{0\}$. 
\end{proof}

From  the definition of the nearly discrete metrics, 
we deduce the next two lemmas. 
\begin{lem}\label{lem:disep}
Let $R$ be a range set.  
Then for all $x, y\in R$, 
the value  $\yodisdis_{R}(x, y)$ coincides 
with the infimum of $\epsilon\in [0, \infty)$ such that 
$x\le y\lor \epsilon$ and $y\le x\lor \epsilon$. 
\end{lem}

\begin{lem}\label{lem:yodisdiscomp}
For every range set $R$, the space 
$(R, \yodisdis_{R})$ is complete. 
\end{lem}

\subsection{Tenuous sets}\label{subsec:tenuous}

A subset $E$ of $[0, \infty)$ is said to be
 \emph{semi-sporadic} if 
there exists a strictly decreasing sequence 
$\{a_{i}\}_{i\in \zz_{\ge 0}}$ in $(0, \infty)$ such that 
$\lim_{i\to \infty}a_{i}=0$ and 
$E=\{0\}\cup \{\, a_{i}\mid i\in \zz_{\ge 0}\, \}$. 
This concept is a half version of 
sporadic sets defined in 
\cite{MR4527953}. 
A subset of $[0, \infty)$ is 
\emph{tenuous} 
if it is finite or semi-sporadic. 
Despite its simplicity, 
this concept has a central role in the present paper.

\subsubsection{Properties of tenuous sets}

\begin{lem}\label{lem:discom}
Let $K$ be a subset of $[0, \infty)$. 
Then 
$K$ is tenuous if and only if 
$K$ is a closed subset of $[0, \infty)$ with respect to the Euclidean topology 
satisfying  that 
  $K\cap [r, \infty)$ is finite
  for all $r\in (0, \infty)$. 
\end{lem}
\begin{proof}
If $K$ is tenuous, 
then $K$ satisfies the conditions stated in the lemma. 
Next assume that $K$ is closed and $K\cap [r, \infty)$ is finite
  for all $r\in (0, \infty)$. 
We only need to consider the case where 
$K$ is infinite. 
We first prove the following claim: 
\begin{enumerate}[label=\textup{(Cl)}]
\item\label{item:claim}
  For all $l\in (0, \infty)$, 
the set $K\cap[0, l)$ has the maximum. 
\end{enumerate}
Indeed, 
due to the fact that  $K$ is infinite and $K\cap [l, \infty)$ is finite, 
 we can take $s\in K\cap [0, l)$. 
Since $K\cap [s, \infty)$ is finite, 
the set $K\cap [0, l)$ has the maximum. 

Using  induction and the claim \ref{item:claim}, we  define a sequence $\{a_{i}\}_{i\in \zz_{\ge 0}}$ as follows: 
The value $a_{0}$ is defined as 
the maximum of $K$.
If $a_{0}, \dots, a_{n}$ have been already determined, 
we define $a_{n+1}$ as the maximum of 
the set $K\cap [0, a_{n})$. 
From the definition of $\{a_{i}\}_{i\in \zz_{\ge 0}}$, 
it follows  that $a_{i+1}<a_{i}$ for all $i\in \zz_{\ge 0}$. 
Put $h=\inf \{\, a_{i}\mid i\in \zz_{\ge 0}\, \}$.
If $h$ were not $0$, then 
the set $K\cap [h, \infty)$ would be infinite. 
This is a contradiction to the assumption. 
Thus $h=0$,  and hence $\lim_{i\to \infty}a_{i}=0$. 
Since $K$ is closed, 
we conclude that 
$K=\{0\}\cup \{\, a_{i}\mid i\in \zz_{\ge 0}\, \}$. 
This means that  $K$ is semi-sporadic, and hence 
it is tenuous. 
\end{proof}

In the nearly discrete space 
$(R, \yodisdis_{R})$, 
compact subsets coincide
with the 
tenuous sets. 
\begin{lem}\label{lem:spsp}
Let $R$ be a range set. 
Then 
a subset $K$ of 
$(R, \yodisdis_{R})$ is compact if and only if  
$K$ is tenuous.  
\end{lem}
\begin{proof}
If $K$ is tenuous, 
then it is compact in $(R, \yodisdis_{R})$. 
Next assume that $K$ is compact. 
We only need to consider the case where 
$K$ is infinite. 
By \ref{item:nd:1} in Lemma \ref{lem:nearlydis}, 
for all  $r\in (0, \infty)$, 
we see that
the set $K\cap [r, \infty)$  is 
discrete and compact, and hence it is 
finite. 
According to Lemma \ref{lem:discom}, 
the set $K$ is tenuous. 
This leads to  the lemma. 
\end{proof}

\begin{cor}\label{cor:sp}
Let $R$ 
be a range set.  
Let $X$ be a compact  space. 
If  $f\colon X\to R$ is  a continuous map with respect to 
$\yodisdis_{R}$, 
then the  image of $f$ is 
a tenuous 
 subset in  
$R$. 
\end{cor}

\subsection{Function spaces}\label{subsec:funcsp}
In this subsection, we 
discuss some properties of 
$(\yomapsp{X}{R}, \yomaindis)$ and 
$(\yomaps{X}{R}, \yomaindis)$.

\subsubsection{Metric propeties}
We first provide two expressions of 
the ultrametric $\yomaindis$ on 
$\yomapsp{X}{R}$ and 
$\yomaps{X}{R}$. 
\begin{prop}\label{prop:supsup}
Let $X$ be a compact space, 
$R$ be a range set, 
and 
$f, g\in \yomapsp{X}{R}$. 
Then  $\yomaindis(f, g) =\sup_{x\in X}\yodisdis_{S}(f(x), g(x))$. 
\end{prop}
\begin{proof}
Put $a=\yomaindis(f, g)$ and 
$b=\sup_{x\in X}\yodisdis_{S}(f(x), g(x))$. 
Take an arbitrary value  $u\in R$ satisfying 
that 
$f(x)\le g(x)\lor u$
and $g(x)\le f(x)\lor u$ for all 
$x\in X$. 
Then, due to 
Lemma \ref{lem:disep},  for each $x\in X$, 
we have $\yodisdis_{S}(f(x), g(x))\le u$. 
This implies that  $b\le u$,  and hence $b\le a$. 

We next show the converse inequality. 
Since $\yodisdis_{S}(f(x), g(x))\le b$ for all 
$x\in X$, 
Lemma \ref{lem:disep} implies  $f(x)\le g(x)\lor b$ and 
$g(x)\le f(x)\lor b$. 
Thus $a\le b$. 
This finishes the proof. 
\end{proof}

For a map $f\colon X\to Y$, 
we write  $\yoimage{f}=f(X)$. 
We can describe $\yomaindis(f, g)$ more precisely. 
\begin{prop}\label{prop:maindisprop}
Let $X$ be a compact space, 
$R$ be a range set, 
and 
$f, g\in \yomapsp{X}{R}$. 
Then the value $\yomaindis(f, g)$ is equal to the 
minimum of all $t\in [0, \infty)$ such that 
for all $x\in X$ with $t<(f\lor g)(x)$, we have $f(x)=g(x)$. 
Moreover, 
the value $\yomaindis(f, g)$ belongs to 
$\yoimage{f}\cup \yoimage{g}\cup\{0\}$. 
Consequently, 
the ultrametric $\yomaindis$  is 
$R$-valued. 
\end{prop}
\begin{proof}
Let $v$ be the infimum of 
 all $t\in [0, \infty)$ such that 
for all $x\in X$ with $t<(f\lor g)(x)$, 
we have $f(x)=g(x)$.
Put $A=\yoimage{f}\cup \yoimage{g}\cup\{0\}$. 
Notice that $A$ is tenuous. 
Suppose, contrary to our statement, 
that 
$v\not\in A$. 
Since $A$ is 
tenuous, we can take 
$u\in [0, \infty)$
such that $u<v$ and $(u, v)\cap A=\emptyset$. 
Then, 
for all $x\in X$ with $u<(f\lor g)(x)$, 
we have $f(x)=g(x)$. 
This implies that $v$ is not the infimum, which 
 is a contradiction. 
Thus $v\in A$. 

We now prove $\yomaindis(f, g)\le v$. 
Take an arbitrary point $x\in X$. 
If $v <(f\lor g)(x)$, 
then the definition of $v$ implies 
 $f(x)=g(x)$, 
 and hence 
we have 
$f(x)\le g(x)\lor v$
 and 
 $g(x)\le f(x)\lor v$. 
If $(f\lor  g)(x)\le v$, 
then $f(x)\le v$ and $g(x)\le v$, 
which  yield    $f(x)\le g(x)\lor v$
and $g(x)\le f(x)\lor v$. 
These two cases  show that  
$\yomaindis(f, g)\le v $. 
To prove $v\le \yomaindis(f, g)$, 
take an arbitrary value
  $t\in R$ with for which 
$f(x)\le g(x)\lor t$ and 
$g(x)\le f(x)\lor t$ for all $x\in X$. 
Take $x\in X$ satisfying that  $t<(f\lor g)(x)$. 
In this setting, 
we may assume that $t< f(x)$. 
Then $g(x)\le f(x)\lor t=f(x)$. 
From 
$f(x)\le g(x)\lor t$ 
and 
$t<f(x)$, 
it follows that 
$f(x)\le g(x)$. 
Thus $f(x)=g(x)$. 
Hence $v\le \yomaindis(f, g)$. 
This completes the proof. 
\end{proof}

\begin{cor}\label{cor:rrr}
Let $X$ be a compact space, 
and $R$ be a range set.  
Let $r\in R\setminus \{0\}$, 
and $f, g\in \yomapsp{X}{S}$. 
Then 
$r=\yomaindis(f, g)$ if and only if 
$f^{-1}(r)\neq g^{-1}(r)$ and
$f(x)=g(x)$
whenever 
$r<(f\lor g)(x)$.  
\end{cor}

Now we notice that all members of the difference 
$\yomapsp{X}{R}\setminus \yomaps{X}{R}$
are isolated
by  Proposition \ref{prop:maindisprop}: 
\begin{lem}\label{lem:isolated}
Let $X$ be a compact space, 
and 
$R$ be a range set. 
If $f\in \yomapsp{X}{R}$ satisfies 
$0\not\in \yoimage{f}$ and we put 
$l=\min \yoimage{f}$, then 
$U(f, l)=\{f\}$. 
In particular, 
the point $f$ is isolated in 
the space $\yomapsp{X}{R}$.
\end{lem}

For a topological  spaces $T$ and for a 
metric space 
$(X, d)$, we denote by 
$\yomapsb{T}{X}$ the set of all 
bounded continuous maps from $T$ to 
$X$. 
Let  $\yobmet_{d}$ denote the supremum metric on 
$\yomapsb{T}{X}$, i.e.,  
$\yobmet_{d}(f, g)=\sup_{x\in X}d(f(x), g(x))$. 

The proof of the  following is 
presented in  
\cite[Theorem 43.6]{MR3728284}. 
\begin{prop}\label{prop:bounded}
Let $T$ be a topological  space 
 and $(X, d)$ be a complete  metric spaces.
Then the space $(\yomapsb{T}{X}, \yobmet_{d})$ is 
complete. 
\end{prop}

Since 
Propositions \ref{prop:supsup}
asserts that $(\yomapsp{X}{R}, \yomaindis)$ is nothing but the space 
$(\yomapsb{T}{X}, \yobmet_{\yodisdis_{R}})$, 
Lemma \ref{lem:yodisdiscomp} and Proposition \ref{prop:bounded}
yield the next proposition:
\begin{prop}\label{prop:funccomplete}
Let $X$ be a compact space, 
and $R$ be a range set. 
Then $(\yomapsp{X}{R}, \yomaindis)$ is 
complete. 
\end{prop}

\begin{cor}\label{cor:comp00}
Let $X$ be a compact space, 
and $R$ be a range set. 
Then $(\yomaps{X}{S}, \yomaindis)$ is 
complete. 
\end{cor}
\begin{proof}
We denote by $I$ the set of all 
$f\in \yomapsp{X}{R}$ 
such that $0\not\in \yoimage{f}$. 
Then we have $\yomaps{X}{R}=
\yomapsp{X}{R}\setminus I$. 
According to 
 Lemma \ref{lem:isolated},  the set 
 $I$ is open, and hence 
the set $\yomaps{X}{R}$ is closed in 
$\yomapsp{X}{R}$. 
Since $\yomapsp{X}{R}$ is complete (see 
Proposition \ref{prop:funccomplete}), 
so is $\yomaps{X}{R}$. 
\end{proof}

\subsubsection{Topological weights of function spaces}\label{subsec:funcweight}
In this subsection, 
we determine the topological weight of 
$(\yomaps{X}{R}, \yomaindis)$.

By the definition of $\yomaps{X}{R}$, we conclude:
\begin{lem}\label{lem:subspace}
If $X$ is a topological space and  range sets $R$ and $S$ satisfy $S\yosub R$, 
then
$\yomaps{X}{S}$ is a 
metric subspace of 
$\yomaps{X}{R}$. 
\end{lem}

For a set $E$, 
we denote by 
$\yoseqsp{E}$
the set of all finite sequence in $E$. 
Notice that if $E$ is infinite, 
then we have $\card(\yoseqsp{E})=\card(E)$.

\begin{prop}\label{prop:mapscard}
Let  
 $X$ be a $0$-dimensional 
 compact Hausdorff  space
 such that 
 $\yoweight{X}=$ is infinite, 
and $R$ be a 
non-characteristic range set with $2\le \card(R)$. 
Then 
$\card(\yomaps{X}{R})=\max\{\yoweight{X}, \card(R)\}$. 
\end{prop}
\begin{proof}
Notice that since $R$ is non-characteristic, 
for all $f\in \yomaps{X}{R}$ the 
image $\yoimage{f}$ is a finite subset of $R$. 
We denote by $\mathcal{O}$
the set of all clopen subsets of $X$. 
Remark that 
$\card(\mathcal{O})=\yoweight{X}$ 
(see 
Lemma \ref{lem:22-22}). 
Put $\tau=\max\{\yoweight{X}, \card(R)\}$. 
We first prove 
$\card(\yomaps{X}{R})\le \tau$. 
For each $f\in \yomaps{X}{R}$, 
put $f(X)=\{0\}\cup 
\{\, a_{f, i}\mid i\in A_{f}\, \}$, 
where 
$A_{f}=\{0, \dots, n\}$ 
for some $n\in \zz_{\ge 0}$. 
We  define 
$K\colon \yomaps{X}{R}\to 
\yoseqsp{\mathcal{O}}\times 
\yoseqsp{R}$ by 
$K(f)=(\{f^{-1}(a_{f, i})\}_{i\in A_{f}}, \{a_{f, i}\}_{i\in A_{f}})$. 
Then $K$ is injective, and hence 
$\card(\yomaps{X}{R})\le 
\max\{\yoweight{X}, \card(\yoseqsp{R})\}= \tau$. 
We next show the converse inequality. 
For each $O\in \mathcal{O}$ and $r\in R\setminus\{0\}$, 
we define $h_{O, r}\in \yomaps{X}{R}$ by 
\[
h_{O, r}(x)=
\begin{cases}
r & \text{if $x\in O$;}\\
0 & \text{if $x\not \in O$.}
\end{cases}
\]
In this situation,  if $O, O^{\prime}\in \mathcal{O}$ and 
$r, r^{\prime}\in R\setminus \{0\}$ satisfy 
$(O, r)\neq (O^{\prime}, r^{\prime})$, then 
we have $h_{O, r}\neq h_{O^{\prime}, r^{\prime}}$. 
Thus we obtain 
$\yoweight{X}\times \card(R)\le \card(\yomaps{X}{R})$. 
Since $\yoweight{X}\times \card(R)=\max\{\yoweight{X}, \card(R)\}$, 
the proof is finished. 
\end{proof}

\begin{lem}\label{lem:finiteimage}
Let 
$X$ be a
$0$-dimensional  compact Hausdorff space
such that $\yoweight{X}$ is infinite, 
and let $R$ be a range set with $2\le \card(R)$,
Then 
the set $F$ of all $f\in \yomaps{X}{R}$ whose images are finite 
is dense in 
$(\yomaps{X}{R}, \yomaindis)$ and 
satisfies 
$\card(F)\le \max\{\yoweight{X}, \card(R)\}$. 
\end{lem}
\begin{proof}
The denseness of $F$ follows from 
Proposition \ref{prop:maindisprop}. 
The  inequality
$\card(F)\le \max\{\yoweight{X}, \card(R)\}$
follows from 
Lemma \ref{lem:leweight}. 
\end{proof}

\begin{prop}\label{prop:mapsweight}
Let 
$X$ be a
$0$-dimensional  compact Hausdorff space
such that  $\yoweight{X}$ is infintie, 
and 
let $R$ be a range set with 
$2\le \card(R)$.  
Then we have 
$\yoweight{\yomaps{X}{R}}=
\max\{\yoweight{X}, \card(R)\}$. 
\end{prop}
\begin{proof}
Put $\tau =\max\{\yoweight{X}, \card(R)\}$. 
Lemmas 
 \ref{lem:weightle} and  \ref{lem:finiteimage}
 yield the inequality
$\yoweight{\yomaps{X}{R}}\le \tau$. 
To prove the converse inequality, 
we divide the proof into two cases. 

Case 1. [$\card(R)\le \aleph_{0}$]: 
Take $r\in R\setminus \{0\}$. 
Since the set $\yomaps{X}{\{0, r\}}$ is 
$r$-separated set of $\yomaps{X}{R}$ 
(see Lemma \ref{lem:subspace}). 
Owing  to the fact that   $\card(\yomaps{X}{\{0, r\}})
=\yoweight{X} =\tau$ (see Proposition \ref{prop:mapscard})
 and 
Lemma \ref{lem:leweight}, 
we see that $\tau\le \yoweight{\yomaps{X}{R}}$, 

Case 2. [$\aleph_{0}<\card(R)$]:
Using  a  partition of $(0, \infty)$ by mutually disjoint  countable intervals, 
the assumption that $\aleph_{0}<\card(R)$ guarantees 
the 
existence of
 $s,t\in (0, \infty)$ such that 
$\card(R\cap (s, t))=\card(R)$. 
Considering the subset $\yomaps{X}{R\cap (s, t)}$, 
as done in Case 1, we
 obtain $\tau\le \yoweight{\yomaps{X}{R}}$. 
\end{proof}

Similarly to Proposition \ref{prop:mapsweight}, 
we have:
\begin{prop}\label{prop:funcballweight}
Let 
$X$ be a $0$-dimensional compact 
Hausdorff space
such that  $\yoweight{X}$ is infinite, 
and
$R$ be a range set and 
$r\in R\setminus \{0\}$. 
If every clopen subset $G$ of $X$ satisfies 
$\yoweight{G}=\yoweight{X}$, 
and either of 
$\card(R)\le \yoweight{X}$ or 
$\card(R\cap [0, r])=\card(R)$
is 
 satisfied, 
then 
for all $f\in \yomaps{X}{R}$, 
we have 
$\yoweight{B(f, r)}=\max\{\yoweight{X}, \card(R)\}$. 
\end{prop}

\subsection{Spaces of ultrametrics}\label{subsec:spultra}
We next consider the  ultrametric space 
$(\yocmet{X}{R}, \umetdis_{X}^{R})$. 

\subsubsection{Metric properties}
First we observe that 
ultrametrics are naturally continuous with respect to 
the nearly discrete topology. 
\begin{prop}\label{prop:spectral}
Let
 $X$ be a topological  space, 
 $R$ be a range set, 
and $d\in \yocmet{X}{R}$. 
Then
$d\colon X^{2}\to R$
is continuous with respect to $\yodisdis_{R}$. 
In particular, 
we have 
$\yocmet{X}{R}\yosub \yomaps{X\times X}{R}$
and $\yomaindis|_{\yocmet{X}{R}}=\umetdis_{X}^{R}$. 
\end{prop}
\begin{proof}
In this proof, 
we consider that 
$X^{2}$ is equipped with the metric  $d\times d$ defined 
 by $(d\times d)((x, y), (a, b))=d(x, a)\lor d(y, b)$. 
Fix    $p=(a, b)\in X^{2}$ and 
$\epsilon \in (0,\infty)$. 
For all $(x, y)\in U(p, \epsilon)$, 
we obtain 
\begin{align*}
d(a, b)&\le d(a, x)\lor d(x, y)\lor d(y, b)\\
&=d(x, y)\lor(d(a, x)\lor d(y, b))\le 
d(x, y)\lor \epsilon. 
\end{align*}
Replacing the role of $(a, b)$ with 
that of $(x, y)$, 
we also obtain 
$d(x, y)\le d(a, b)\lor \epsilon$. 
From Lemma \ref{lem:disep}, 
it follows that 
$\yodisdis_{R}(d(x, y), d(a, b))\le \epsilon$. 
Therefore
$d\colon X^{2}\to R$ 
is continuous with 
respect $\yodisdis_{R}$. 
\end{proof}

\begin{cor}\label{cor:spec}
Let $R$ be a range set, 
 $X$ be a compact  space, 
and $d\in \yocmet{X}{R}$. 
Then the image of $d$ is tenuous. 
\end{cor}

\begin{rmk}
Corollary \ref{cor:spec} has been  already known, see for example, 
\cite[The statement (1) in Theorem 1.7]{MR3782290}
and \cite[(v) in Proposition 19.2]{MR2444734}. 
Proposition \ref{prop:spectral} gives 
another explanation of this phenomenon on images of ultrametrics on compact spaces. 
Namely, that property can be reduced 
to  the fact that
ultrametrics are continuous with respect to not only 
the Euclidean topology but also the nearly discrete topology.
\end{rmk}

Similarly to Proposition \ref{prop:funccomplete}, 
we obtain:
\begin{prop}\label{prop:spmetcomp}
Let $X$ be a 
compact  space, 
and $R$ be a range set. 
Then the space 
$(\yocmet{X}{R}, \umetdis_{X}^{R})$ is 
complete. 
\end{prop}
\begin{proof}
Take a Cauchy sequence $\{d_{i}\}_{i\in \zz_{\ge 0}}$
in
$(\yocmet{X}{R}, \umetdis_{X}^{R})$. 
By Proposition \ref{prop:funccomplete},   we obtain a limit $d\in \yomaps{X^{2}}{R}$
with respect to $\yomaindis$. 
Since $d$ is also a point-wise limit of $\{d_{i}\}_{i\in \zz_{\ge 0}}$ with respect to the Euclidean topology, 
we conclude that $d$ satisfies the strong triangle inequality and $d(x, y)=d(y, x)$ for all $x, y\in X$. 
This implies that 
$d\in \yocmet{X}{R}$. 
\end{proof}

\begin{df}
Let $(X, d)$ be an ultrametric space, and 
$r\in (0, \infty)$. 
For a point $p\in X$
we denote by 
$\yocutmem{p}{d}{r}$ the 
closed ball $B(p, r; d)$. 
We also denote by 
$\yocutsp{X}{d}{r}$
the set
$\{\, \yocutmem{p}{d}{r}\mid p\in X\, \}$. 
We define a metric
$\yocutdis{d}{r}$
 on $\yocutsp{X}{d}{r}$ by 
$\yocutdis{d}{r}(\yocutmem{p}{d}{r}, \yocutmem{q}{d}{r})=d(p, q)$. 
Lemma \ref{lem:balldis} guarantees the well-definedness of 
$\yocutdis{d}{r}$. 
The space $(\yocutsp{X}{d}{r}, \yocutdis{d}{r})$ 
is an analog of the 
\emph{$t$-closed quotient} of a metric space defined in 
\cite{MR4462868}. 
The symbol ``$\mathfrak{c}$'' involved in the notations stands for the 
``closed''. 
\end{df}

As a consequence of 
Propositions \ref{prop:maindisprop}
 and \ref{prop:spectral}, 
we obtain the next lemma, 
which is an 
analogue of 
\cite[Theorem 5.1]{MR4462868} for spaces of ultrametrics.

\begin{cor}\label{cor:spectrul}
Let $X$ be a $0$-dimensional  compact 
Hausdorff space. 
For all $d, e\in \yocmet{X}{R}$, 
the value $\umetdis_{X}^{R}(d, e)$ is 
equal to 
the minimum $r\in R$ such that 
\begin{enumerate}[label=\textup{(\arabic*)}]
\item 
we have 
$\yocutsp{X}{d}{r}=\yocutsp{X}{e}{r}$. 
Namely,  
$\{B(a, r; d)\mid a\in X\}=\{B(a, r; e)\mid a\in X\}$;
\item 
we have 
$\yocutdis{d}{r}=\yocutdis{e}{r}$. 
\end{enumerate}
Moreover, the ultrametric $\umetdis_{X}^{R}$ is 
$R$-valued. 
\end{cor}
\subsubsection{Results on dense subsets}
We review the  author's  theorem of dense subsets of 
spaces of ultrametrics. 
In particular, 
Theorem \ref{thm:densefrac} is a new result. 
\begin{df}\label{df:quintuple}
For a range set  $R$, 
we say that a quintuple 
$\yoquin=(X, I, r,  \{B_{i}\}_{i\in I}, \{e_{i}\}_{i\in I})$ is 
an \emph{$R$-amalgamation system} if 
the following conditions are satisfied:
\begin{enumerate}
\item the symbol $X$ is a topological space, and 
 $I$ is a set (of indices);
\item we have $r\in \ult{I}{R}$, where we consider that 
$I$ is equipped with the discrete topology; 
\item the family $\{B_{i}\}_{i\in I}$ is a covering of 
$X$ consisting of mutually disjoint clopen subsets of 
$X$;
\item for each $i\in I$, 
we have  $e_{i}\in \yocmet{B_{i}}{R}$. 
\end{enumerate}
Moreover, if the system $\yoquin$ satisfies 
\begin{enumerate}[label=\textup{(Pl)}]
\item for each $i\in I$, we have  $e_{i}\in \ult{B_{i}}{R}$, 
\end{enumerate}
then  the system is said to be 
\emph{pleasant}. 
\end{df}

For a metric space $(X, d)$, and a subset $A$ of $X$, 
we denote by $\yodiam_{d}(A)$ the diameter of 
$A$ with respect to $d$. 
The proof of 
the next proposition is 
essentially 
presented in \cite[Proposition 3.1]{Ishiki2021dense} and \cite[Proposition 2.1]{Ishiki2023disco}. 
\begin{prop}\label{prop:amalgam}
Let 
$X$ be a topological space,  
 $R$ be a range set,
and  $\mathfrak{A}=(X, I, r, \{B_{i}\}_{i\in I}, \{e_{i}\}_{i\in I})$ be an 
$R$-amalgamation system. 
Define a function  $D:X^2\to [0, \infty)$ by 
\begin{align}\label{al:metric}
D(x, y)=
\begin{cases}
e_i(x, y) & \text{ if $x, y\in B_i$;}\\
e_i(x, p_i)\lor r(i, j)\lor e_j(p_j, y)  & \text{if $x\in B_i$ and $y\in B_j$.}
\end{cases}
\end{align}
Then $D$ belongs to $\yocmet{X}{R}$. 
Moreover, the following statements are true. 
\begin{enumerate}[label=\textup{(\arabic*)}]
\item\label{item:amal:1}
If $X$ is ultrametrizable and $\mathfrak{A}$ is pleasant, 
then $D\in \ult{X}{R}$. 
\item\label{item:amal:2}
If $d\in \yocmet{X}{R}$ satisfies 
$r(i, j)=d(p_{i}, p_{j})$, and 
if $\epsilon \in (0, \infty)$ satisfies that 
$\yodiam_{d}(B_i)\le \epsilon$
and 
$\yodiam_{e_i}(B_i)\le \epsilon$
for all $i\in I$, 
then 
$\mathcal{UD}_{X}^S(D, d)\le \epsilon$. 
\end{enumerate}
\end{prop}
Based on Proposition
\ref{prop:amalgam}, 
the function  defined in \eqref{al:metric} is called the 
\emph{pseudo-ultrametric (ultrametric) associated with $\mathfrak{A}$}.

We next investigate dense subsets of 
spaces of 
continuous pseudo-ultrametrics. 
\begin{prop}\label{prop:yocsetdense}
Let $X$ be a $0$-dimensional 
compact
Hausdorff 
space, 
and 
$R$ be a characteristic  range set. 
Then 
$\ult{X}{R}$ is dense in 
$(\yocmet{X}{R}, \umetdis_{X}^{R})$. 

\end{prop}
\begin{proof}
Take  $d\in \yocmet{X}{R}$ and 
$\epsilon \in (0, \infty)$. 
By the compactness of $X$, we can find 
 finite points $\{p_{i}\}_{i\in I}$, 
where $I=\{0, \dots, k\}$, 
for which 
the family $\{B(p_{i},\epsilon; d)\}_{i\in I}$
is 
a mutually disjoint covering of $X$. 
We define a metric $r$ on $I$ by 
$r(i, j)=d(p_{i}, p_{j})$. 
For each $i\in I$, 
using 
Proposition \ref{prop:equivultmet}
take 
$e_{i}\in \ult{B(p_{i}, \epsilon; d)}{R}$. 
Then  $\yoquin=(X, I, r,  \{B(p_{i}, \epsilon)\}_{i\in I}, \{e_{i}\}_{i\in I})$ is a pleasant  $R$-amalgamation system. 
Let $D$ denote the pseudo-metric associated with $\yoquin$. 
Then the 
statements \ref{item:amal:1} and 
\ref{item:amal:2} 
 of  Proposition \ref{prop:amalgam} imply that 
$D\in \ult{X}{R}$ and 
$\umetdis_{X}^{R}(d, D)\le \epsilon$. 
Therefore $\ult{X}{R}$ is dense 
in $\yocmet{X}{R}$. 
\end{proof}

The following theorem is a 
reformulation of 
\cite[Theorem 1.3]{Ishiki2021dense}.
\begin{thm}\label{thm:doublingdense}
Let  $X$ be an ultrametrizable space, 
and 
 $R$ be a characteristic 
range set. 
Then $X$ is compact if and only if the set of all 
doubling ultrametrics is dense $F_{\sigma}$ in 
$(\ult{X}{R}, \umetdis_{X}^{R})$. 
\end{thm}

A range set $R$ is \emph{quasi-complete} if 
there exists $C\in [1, \infty)$ such that 
for every bounded subset $E$ of $R$, 
there exists $s\in R$ with 
$\sup E \le s \le C\cdot \sup E$. 
The next theorem  is stated in 
\cite[Theorem 6.7 and Proposotion 6.9]{Ishiki2021ultra}. 
\begin{thm}\label{thm:non-ddense}
Let  $X$ be a ultrametrizable space with an accumulation point, 
and $R$ be an quasi-complete characteristic  
range set. 
Then the set of all non-doubling ultrametrics in 
$\ult{X}{R}$ is dense $G_{\delta}$ in 
$(\ult{X}{R}, \umetdis_{X}^{R})$. 
\end{thm}

A range set is \emph{exponential} if
there exist $a\in (0, 1)$ and 
$M\in [1, \infty)$ such that  for 
every $n\in \zz_{\ge 0}$  we have 
$[M^{-1}a^{n}, Ma^{n}]\cap R\neq \emptyset$. 
The proof of the next theorem is presented in 
\cite[Theorem 1.5]{Ishiki2021dense}. 
\begin{thm}\label{thm:updense}
Let $R$ be a characteristic
range set. 
The the following statements are true: 
\begin{enumerate}
\item 
The set $R$ is exponential if and only if 
the set of all uniformly perfect ultrametrics in 
$\ult{X}{R}$ is dense $F_{\sigma}$ in 
the space
$(\ult{X}{R}, \umetdis_{X}^{S})$
\item The set of all non-uniformly perfect ultrametrics is 
dense $G_{\delta}$ in $(\ult{X}{R}, \umetdis_{X}^{R})$. 
\end{enumerate}
\end{thm}

For a metric space 
$(X, d)$, 
we denote by $\yorhdim{X, d}$, 
$\yorpdim{X, d}$, 
$\yorubdim{X, d}$, 
and 
$\yoradim{X, d}$ 
the  Hausdorff dimension, 
the packing dimension, 
the upper box dimension, 
and 
the Assouad dimension of $(X, d)$.
For the definitions and details  of these 
dimensions, 
 we refer the readers to
  \cite{falconer1997techniques}, 
  \cite{falconer2004fractal}, 
 \cite{fraser2020assouad}, 
  \cite{Ishiki2023fractalin}, 
  and 
    \cite{Ishiki2022factor}.
We also denote by $\yorlset$
the set of all $(a_{1}, a_{2}, a_{3}, a_{4})\in [0, \infty]^{4}$
such that $a_{1}\le a_{2}\le a_{3}\le a_{4}$. 
Let $X$ be a metrizable space, and 
$R$ be a range set. 
For $\yorbd{a}=(a_{1}, a_{2}, a_{3}, a_{4})\in \yorlset$, 
let $\yordimsetu{X}{R}{\yorbd{a}}$ 
stand for  the 
set of all bounded metric $d\in \ult{X}{R}$ 
satisfying that
$\yorhdim{X, d}=a_{1}$, 
$\yorpdim{X, d}=a_{2}$, 
$\yorubdim{X, d}=a_{3}$, 
and 
$\yoradim{X, d}=a_{4}$. 

Even if the readers does not know the 
details of the dimensions explained above, 
all arguments in this paper 
are traceable  using only the finite stability 
 of 
dimensions 
stated in the next lemma. 
The proof can be found in 
\cite{falconer1997techniques}, 
\cite{falconer2004fractal}, and \cite{fraser2020assouad}.

\begin{lem}\label{lem:finstab}
Let $\mathbf{Dim}$ be any one of 
$\hdim$, 
$\pdim$, 
$\ubdim$, 
and $\adim$. 
If $(X, d)$ is a metric space, and 
$A$ and $B$ are subsets of $X$, 
then we have 
$\mathbf{Dim}(A\cup B, d)
=\max\{\mathbf{Dim}(A, d), 
\mathbf{Dim}(B, d)\}$. 
\end{lem}

\begin{df}\label{df:ad}
We denote by $\yocantorc$ the 
Cantor set. 
For $\yorbd{a}\in \yorlset$, 
a characteristic range set $R$ is 
said to be 
\emph{$(\yocantorc, \yorbd{a})$-admissible} if 
$\yordimsetu{\yocantorc}{R}{\yorbd{a}}\neq \emptyset$. 
\end{df}
For example, the 
characteristic range sets $[0, \infty)$ and $\qq_{\ge 0}$
are $(\yocantorc, \yorbd{a})$-admissible for all 
$\yorbd{a}\in\yorlset$.

We define 
$\yorzeron=(0, 0, 0, 0)\in \yorlset$.
The next lemma is obtained by 
\cite[Lemma 3.7]{Ishiki2023fractalin}
and 
the 
fact that we can take a strictly decreasing sequence  convergent to $0$
from 
every  characteristic range set. 
\begin{lem}\label{lem:0000}
Every characteristic range set is 
$(\yocantorc, \yorzeron)$-admissible. 
\end{lem}

\begin{lem}\label{lem:r:r}
For every $\yorbd{a}\in \yorlset$, 
if a characteristic range set $R$ is 
$(\yocantorc, \yorbd{a})$-admissible, 
then so is 
$R\cap [0, r]$ 
for all $r\in R\setminus \{0\}$.  
\end{lem}
\begin{proof}
Take $d\in \yordimsetu{\yocantorc}{R}{\yorbd{a}}$
and define $e\in \ult{\yocantorc}{R\cap [0, r]}$ by 
 $e(x, y)=\min\{d(x, y), r\}$. 
 By the compactness of $\yocantorc$, 
 we can find points  $\{p_{i}\}_{i=0}^{k}$ in $\yocantorc$
 for which $\bigcup_{i=0}^{k}B(p_{i}, r/2; e_{i})=X$. 
 From the definition of $e_{i}$, it follows that 
 $B(p_{i}, r/2; e_{i})=B(p_{i}, r/2; d)$. 
Thus, 
according to  Lemma \ref{lem:finstab} and 
$d\in \yordimsetu{\yocantorc}{R}{\yorbd{a}}$, 
we have $e\in \yordimsetu{\yocantorc}{R\cap[0, r]}{\yorbd{a}}$. 
\end{proof}

By 
\cite[The statement (N3) in Theorem 4.7]{Ishiki2022factor}, 
we obtain the next extension theorem of ultrametrics 
preserving fractal dimensions. 
\begin{lem}\label{lem:frac1}
Let $X$ be a separable ultrametrizable space, 
$R$ be a characteristic range set, 
and
$F$ be a closed subset of $X$. 
Then for all  $\yorbd{a}\in \yorlset$ and 
for all 
$d\in \yordimsetu{F}{R}{\yorbd{a}}$, 
there exists 
$D\in \yordimsetu{X}{R}{\yorbd{a}}$
such that $D|_{A^{2}}=d$.
\end{lem}

\begin{lem}\label{lem:dimext}
If 
$X$ is an uncountable compact  ultrametrizable 
space,
$\yorbd{a}\in \yorlset$, 
and 
a characteristic range set $R$ is a $(\yocantorc, \yorbd{a})$-admissible,  then
$\yordimsetu{X}{R}{\yorbd{a}}\neq \emptyset$. 
\end{lem}
\begin{proof}
By 
\cite[Corollary 6.5]{MR1321597}, 
there exists a closed subspace $F$ of $X$ that is 
homeomorphic to the Cantor set. 
Since $R$ is  $(\yocantorc, \yorbd{a})$-admissible, 
we can take 
$d\in \yordimsetu{F}{R}{\yorbd{a}}$. 
Applying 
Lemma \ref{lem:frac1} to 
$X$, $F$, and $d$, 
we obtain $D\in \yordimsetu{X}{R}{\yorbd{a}}$.
\end{proof}

We next
verify that $\yordimsetu{X}{R}{\yorbd{a}}$ is 
dense in $\ult{X}{R}$. 
\begin{thm}\label{thm:densefrac}
If 
 $X$ is  an uncountable  
compact ultrametrizable space, 
 $\yorbd{a}\in \yorlset$, 
 and 
$R$ is  a $(\yocantorc, \yorbd{a})$-admissible 
characteristic range set, 
then the set $\yordimsetu{X}{R}{\yorbd{a}}$ is 
dense in $\ult{X}{R}$. 
\end{thm}
\begin{proof}
Take $d\in \ult{X}{R}$ and $\epsilon\in (0, \infty)$.
We also take $\eta\in R\setminus \{0\}$ with 
$\eta\le \epsilon$ and 
define $S_{\eta}=R\cap [0, \eta]$. 
Then Lemma \ref{lem:r:r} implies 
that $S_{\eta}$ is $(\yocantorc, \yorbd{a})$-admissible. 
Since $X$ is compact, 
there exists a sequence  $\{p_{i}\}_{i=0}^{n}$ such that 
the family $\{B(p_{i}, \epsilon)\}_{i=0}^{n}$ is 
a mutually disjoint covering of $X$. 
Put $I=\{0, \dots, n\}$. 
In this setting, 
each $B(p_{i}, \epsilon)$ is also 
compact and  ultrametrizable. 
Since $X$ is uncountable, 
there exists $k\in I$ such that 
$B(p_{k}, \epsilon)$ is uncountable. 
Lemma \ref{lem:dimext} enables us to 
take 
$e_{k}\in 
\yordimsetu{B(p_{k}, \epsilon)}{S_{\eta}}{\yorbd{a}}$. 
For each 
$i\in I\setminus \{k\}$, 
we have take 
$e_{i}\in \yordimsetu{B(p_{i}, \epsilon)}{S_{\eta}}{\yorzeron}$.
We define a metric  $h$ on $I$ by 
$h(i, j)=d(p_{i}, p_{j})$. Note that 
$h$ generates the discrete topology on $I$.
Then  $\yoquin 
=(X, I, h, \{B(p_{i}, \epsilon)\}_{i\in I}, \{p_{i}\}_{i\in I})$ is 
a
pleasant $R$-amalgamation system. 
Let $D$ be the ultrametric associated with $\yoquin$. 
By \ref{item:amal:2} in Proposition 
\ref{prop:amalgam}, we have 
$\umetdis_{X}^{R}(d, D)\le \epsilon$. 
Due to $e_{k}\in \yordimsetu{B(p_{k}, \epsilon)}{S_{\eta}}{\yorbd{a}}$
and 
$e_{i}\in \yordimsetu{B(p_{i}, \epsilon)}{S_{\eta}}{\yorzeron}$ for all $i\in I\setminus \{k\}$, 
Lemma \ref{lem:finstab} implies 
$D\in \yordimsetu{X}{R}{\yorbd{a}}$. 
This finishes the proof. 
\end{proof}

\subsubsection{Topological weights of spaces of ultrametrics}
We next estimate the 
topological weights of spaces of continuous pseudo-ultrametrics. 
Most of arguments
 in this subsection are 
analogous with  Subsection \ref{subsec:funcweight}. 
By the definition of $\yocmet{X}{R}$, we notice the following lemma:
\begin{lem}\label{lem:sbsbspsp}
If $X$ is a topological space and 
range sets $R$ and $S$ satisfy 
$S\yosub R$, then 
$\yocmet{X}{S}$ is a metric subspace 
of $\yocmet{X}{R}$. 
\end{lem}

\begin{lem}\label{lem:metricscard}
Let 
 $X$ be $0$-dimensional compact 
Hausdorff space
such that 
 $\yoweight{X}$ is infinite, 
and 
$R$ be a non-characteristic range set
with $2\le \card(R)$.
Then we have 
$\card(\yocmet{X}{R})=
\max\{\yoweight{X}, \card(R)\}$. 
\end{lem}
\begin{proof}
Put $\tau=\max\{\yoweight{X}, \card(R)\}$. 
Let $\mathcal{O}$ be the 
set of all clopen subsets of $X$. 
Similarly to Proposition \ref{prop:mapscard},  we have 
$\card(\yocmet{X}{R})\le \tau$. 
We next show the converse inequality. 
Let $\mathcal{H}$ be a subset of 
$\mathcal{O}$ such that 
\begin{enumerate}
\item 
if $O\in \mathcal{H}$, then $X\setminus O\not \in \mathcal{H}$;
\item for every $O\in \mathcal{O}$, 
we have $O\in \mathcal{H}$ or $X\setminus O\in \mathcal{H}$.
\end{enumerate}
Notice that $\card(\mathcal{H})=\yoweight{X}$. 
For each $O\in \mathcal{H}$, 
and $r\in R\setminus \{0\}$, 
we define $u_{O, r}\in \yocmet{X}{R}$ by 
\[
u_{O, r}=
\begin{cases}
r  & \text{if $x, y\in O$, or $x, y\in X\setminus O$;}\\
0 & \text{otherwise.}
\end{cases}
\]
If $O, O^{\prime}\in \mathcal{H}$ and 
$r, r^{\prime}\in R\setminus \{0\}$
satisfies $(O, r)\neq (O^{\prime}, r^{\prime})$, 
then $u_{O, r}\neq u_{O^{\prime}, r^{\prime}}$. 
Therefore
 we have $\tau \le \card(\yocmet{X}{R})$. 
\end{proof}

Corresponding to 
Lemma \ref{lem:metricscard}, 
we can prove the next lemma:
\begin{lem}\label{lem:metinfiniteimage}
Let $X$ be a $0$-dimensional 
compact Hausdorff space 
such that $\yoweight{X}$ is infinite, 
and $R$ be a range set. 
Let $G$ be  the set of all 
$d\in \yocmet{X}{R}$
such that $\yoimage{d}$ is finite. 
Then the set $G$ is 
dense in $\yocmet{X}{R}$ and  satisfies 
$\card(G)\le \max\{\yoweight{X}, \card(R)\}$
\end{lem}

The following proposition can be proven in 
the same  way to 
Proposition 
\ref{prop:mapsweight}. 

\begin{prop}\label{prop:metricweight}
Let
$X$ be a $0$-dimensional 
compact Hausdorff space 
such that 
$\yoweight{X}$ is infinite, 
and 
 $R$ be a range set. 
Then 
$\yoweight{\yocmet{X}{R}}=
\max\{\yoweight{X}, \card(R)\}$. 
\end{prop}

Similarly to
Proposition 
\ref{prop:metricweight}, 
we obtain: 
\begin{prop}\label{prop:metballweight}
Let 
$X$ be a $0$-dimensional compact 
Hausdorff space
such that 
 $\yoweight{X}$ is infinite, 
$R$ be a range set and 
$r\in R\setminus \{0\}$. 
If every clopen subset $G$ of $X$ satisfies 
$\yoweight{G}=\yoweight{X}$, 
and either of 
$\card(R)\le \yoweight{X}$ or 
$\card(R\cap [0, r])=\card(R)$
 satisfied, 
then 
for all $d\in \yocmet{X}{R}$, 
we have 
$\yoweight{B(d, r)}=\max\{\yoweight{X}, \card(R)\}$. 
\end{prop}

\subsection{Injective ultrametric spaces}\label{subsec:inj}
We show some properties of 
injective ultrametric spaces. 
For a class $\yochara{C}$ of finite metric spaces, 
we say that  a metric space $(X, d)$ satisfies 
the \emph{one-point extension property 
for $\yochara{C}$} if 
for every finite metric space 
$(A\sqcup \{\omega\}, e)$ in $\yochara{C}$ 
(with the specific point $\omega$), 
and for every isometric embedding 
$\phi\colon A\to X$, 
there exists a point $q\in X$ such that 
$d(\phi(a), q)=e(a, \omega)$ for all $a\in A$. 
Using induction, 
we obtain: 
\begin{lem}\label{lem:optext}
Let $\yochara{C}$ be a
class of finite metric spaces. 
A metric space $(X, d)$  is 
$\yochara{C}$-injective if and only if it satisfies the 
one-point extension property for $\yochara{C}$. 
\end{lem}

The next lemma gives a more practical  
condition for the injectivity. 
\begin{lem}\label{lem:equivuniv}
Let $\yochara{C}$ be a class of 
finite
ultrametric spaces. 
An ultrametric space $(X, d)$ is 
$\yochara{C}$-injective if and only if 
it satisfies the following property:
\begin{enumerate}[label=\textup{(Ex)}]
\item\label{item:property}
Let $(A\sqcup\{\omega\}, e)$ be an arbitrary 
finite metric space in $\yochara{C}$ and
$\phi\colon A\to X$ be an isometric map.  
Let $p\in A$ be a point satisfying that 
$e(p, \omega)=\min_{a\in A}e(a, \omega)$. 
Then 
there exists a point $q\in X$ such that 
if $a\in A$ 
satisfies  $e(a, \omega)=e(p, \omega)$, 
then $e(a, \omega)=d(\phi(a), q)$. 
\end{enumerate}
\end{lem}
\begin{proof}
It suffices to show that  if 
$(X, d)$ satisfies the 
property \ref{item:property}, 
then it is 
$\yochara{C}$-injective.
Let $(A\sqcup\{\omega\}, e)$ be 
an ultrametric space in $\yochara{C}$, and 
$\phi\colon A\to X$ be an isometric embedding. 
Using the property \ref{item:property}, we obtain 
$q\in X$ such that 
if $a\in A$ 
satisfies  $e(a, \omega)=e(p, \omega)$, 
then $e(a, \omega)=d(\phi(a), q)$. 
Notice that $e(p, \omega)=d(\phi(p), q)$. 
To show that $(X, d)$ satisfies the 
one-point extension property for 
$\yochara{C}$, take $a\in A$. 
If $e(p, \omega)=e(a, \omega)$, 
then the property of $q\in X$ implies 
$e(a, \omega)=d(\phi(a), q)$. 
If  $e(p, \omega)<e(a, \omega)$, 
then 
Lemma \ref{lem:isosceles}
implies  $e(p, a)=e(a, \omega)$. 
Due to  $d(\phi(p), q)=e(p, \omega)$
and $d(\phi(p), \phi(a))=e(p, a)$, 
we have $d(\phi(p), q)<d(\phi(p), \phi(a))$. 
According to  Lemma  
\ref{lem:isosceles} again, 
we obtain  $d(\phi(a), q)=d(\phi(a), \phi(p))$. 
From  $e(p, a)=e(a, \omega)$ and 
$e(a, \omega)=d(\phi(a), \phi(p))$, it follows that 
$d(\phi(a), q)=e(a, \omega)$. 
In any case, we obtain $d(\phi(a), q)=e(a, \omega)$. 
Therefore  $(X, d)$ satisfies the 
one-point extension property for $\yochara{C}$, 
and hence it is $\yochara{C}$-injective.
\end{proof}

The injectivity is hereditary to 
dense subsets. 
The next lemma can be proven in  a similar way to 
\cite[Proposition 2.4]{MR2754373}. 
\begin{lem}\label{lem:denseuniv}
Let $R$ be a range set,
and
 $(X, d)$ be an $\youfin(R)$-injective ultrametric space.  
If  $H$ is  a dense subset of $X$, 
then 
 $(H, d)$ is $\youfin(R)$-injective. 
\end{lem}

\section{Function spaces}\label{sec:funcsp}
In this section, we explain our first construction of 
injective ultrametric space. 
\begin{df}\label{df:attach}
Let $X$ be a topological space, 
$R$ be a range set, 
and $E$ a subset of $\yomaps{X}{R}$. 
For 
a non-empty subset $A$ of $E$, 
for  $\zeta\in A$, 
and for 
$r\in R\setminus \{0\}$, 
a subset $E$ is said to be 
\emph{$(A, \zeta, r)$-attachable}
if 
there exists 
$g\in E$ such that 
\begin{enumerate}[label=\textup{(A\arabic*)}]
\item\label{item:www}
for all $h\in A$ with $\yomaindis(h, \zeta)\le r$, we have $g^{-1}(r)\neq h^{-1}(r)$ ;
\item\label{item:www2}
we have $g(x)=\zeta(x)$
whenever 
$x\in X$ satisfies 
$r<(\zeta\lor g)(x)$. 
\end{enumerate}
We say that $E$ is \emph{full-attachable} if
it is $(A, \zeta, r)$-attachable for every 
finite subset
$A$ of $E$, for every $\zeta\in E$, and for every 
$r\in R\setminus \{0\}$. 
\end{df}

\begin{rmk}
According to Corollary 
\ref{cor:rrr}, 
the pair of 
the conditions \ref{item:www} and 
\ref{item:www2}
is   equivalent to the property that 
 if $h\in A$ satisfies $\yomaindis(h, \zeta)\le r$, 
 then $\yomaindis(\zeta, g)=r$. 
\end{rmk}

\begin{thm}\label{thm:favoid}
Let $R$ be a
range set, 
 $X$ be a $0$-dimensional Hausdorff  space.
A subset $E$ of $(\yomaps{X}{S}, \yomaindis)$
is $\youfin(R)$-injective if and 
only if 
$E$ 
is full-attachable. 
\end{thm}
\begin{proof}
First assume that $E$ is $\youfin(R)$-injective. 
Take an arbitrary non-empty subset $A$ of $E$, 
$\zeta\in A$, and $r\in R\setminus \{0\}$. 
Put $B=A\sqcup\{\omega\}$ and
define a metric $e$ on $B$
by $e|_{A^{2}}=d$ and 
$e(x, \omega)=e(x, \zeta)\lor r$ for all
$x\in A$. 
We also define an isometric embedding 
$\phi\colon A\to E$ by $\phi(a)=a$. 
Since $E$ is $\youfin(R)$-injective, 
 we can take $g\in E$ such that 
$e(a, \omega)=\yomaindis(\phi(a), g)$ for all $a\in A$.
By Proposition 
\ref{prop:maindisprop} and  
Corollary \ref{cor:rrr}, 
we conclude  that $E$ is full-attachable. 

Next assume that $E$ is 
full-attachable. 
Let $(F\sqcup\{\omega\}, e)$ be an arbitrary 
finite ultrametric space
in $\youfin(R)$,  
and 
$\phi\colon F\to E$ be an isometric embedding. 
Put 
$r=\min_{a\in F}e(a, \omega)$, and take 
$p\in F$ with $r=e(p, \omega)$. 
Since $E$ is $(\phi(F), \phi(p), r)$-attachable, 
we can find 
$g\in E$ satisfying the properties 
\ref{item:www} and \ref{item:www2}. 
To show that $E$ satisfies the property \ref{item:property} in Lemma \ref{lem:equivuniv}, 
take $a\in F$ with $e(a, \omega)=e(p, \omega)(=r)$. 
Then $e(a, p)\le r$,  and hence 
$\yomaindis(\phi(a), \phi(p))\le r$. 
According to \ref{item:www}, 
we have $\phi(a)^{-1}(r)\neq g^{-1}(r)$. 
On account of Proposition \ref{prop:maindisprop}, 
the  inequality $\yomaindis(\phi(a), \phi(p))\le r$
shows that  $\phi(a)(x)=\phi(p)(x)$ whenever 
$r< (\phi(a)\lor \phi(p))(x)$. 
Due to the property \ref{item:www2}, 
we also have
$g(x)=\phi(p)(x)$ whenever 
$r<(g\lor \phi(p))(x)$. 
Therefore, for all $x$ such that 
$r<(g\lor \phi(a))(x)$, 
we obtain $g(x)=\phi(a)(x)$. 
By 
Corollary \ref{cor:rrr}, 
we notice that $\yomaindis(\phi(a), g)=r$.
Namely, 
the space $E$ satisfies the the property \ref{item:property} in Lemma \ref{lem:equivuniv}, 
and hence 
 we conclude that  $E$ is
$\youfin(R)$-injective. 
This finishes the proof of the theorem. 
\end{proof}

We shall  provide examples of 
$\youfin(S)$-injective ultrametric spaces for a 
range set $R$.

\begin{thm}\label{thm:stand}
Let $R$ be a range set, 
and $X$ be a $0$-dimensional 
compact Hausdorff space without  isolated points. 
Then the $R$-valued 
ultrametric space 
$(\yomaps{X}{R}, \yomaindis)$ is 
$\youfin(R)$-injective and complete. 
\end{thm}

\begin{proof}
According  to Proposition \ref{prop:funccomplete} and 
Theorem \ref{thm:favoid}, 
we only need to show that 
$(\yomaps{X}{R}, \yomaindis)$ is 
full-attachable. 
Take a non-empty finite subset $A$ of $\yomaps{X}{R}$, 
 $\zeta\in A$, and $r\in R\setminus \{0\}$. 
Put $B=\zeta^{-1}([0, r])$. 
Then $B$ is clopen in $X$. 
Since $X$ is infinite and it has no isolated points, 
we can find  non-empty clopen subsets $K$ and $L$ of $\yomaps{X}{R}$ such that 
$K\cap L=\emptyset$ and  $K\cup L=B$. 
Notice that $K\neq B$. 
We define a map $g\colon X\to R$ by 
\[
g(x)=
\begin{cases}
\zeta(x) & \text{if $x\not\in B$;}\\
r & \text{if $x\in K$;}\\
0 &\text{if $x\in L$.}
\end{cases}
\]
Then $g$ satisfies the properties 
\ref{item:www} and \ref{item:www2}. 
Therefore the space 
$(\yomaps{X}{R}, \yomaindis)$ is 
full-attachable, and hence it is 
$\youfin(R)$-injective. 
\end{proof}

\begin{thm}\label{thm:notinj}
Let $X$ be a $0$-dimensional compact Hausdorff space, 
and $R$ be a
range set. 
If
$3\le \card(R)$, and 
 $X$ has an isolated point, 
 then 
$(\yomaps{X}{R}, \yomaindis)$ is not 
$\youfin(R)$-injective. 
\end{thm}
\begin{proof}
Take $r, l\in R$ with $0<r<l$, and
let $p\in X$ be an isolated point of $X$, i.e., 
the singleton $\{p\}$ is open. 
Define $\zeta\colon X\to R$ by 
\[
\zeta(x)=
\begin{cases}
l & \text{if $x\in X\setminus \{p\}$;}\\
0 & \text{if $x=p$.}
\end{cases}
\]
Then $\zeta\in \yomaps{X}{R}$. 
In this setting, 
the space $(\yomaps{X}{R}, \yomaindis)$ is 
not $(\{\zeta\}, \zeta, r)$-attachable. 
Indeed, if there exists $g\in \yomaps{X}{R}$
satisfying the properties \ref{item:www} and 
\ref{item:www2}, 
then $g^{-1}(a)\neq \emptyset$, 
$g^{-1}(0)\neq \emptyset$ and 
$g(x)=\zeta(x)$ for all $x\in X\setminus \{p\}$. 
Then $g$ should satisfy $g^{-1}(r)=\{p\}$ and 
$g^{-1}(0)=\{p\}$. This is impossible. 
\end{proof}

Combining 
Theorems \ref{thm:stand} and 
\ref{thm:notinj}, 
we obtain 
a characterization of 
spaces without isolated points. 

\begin{cor}\label{cor:charaiso}
Let $X$ be a $0$-dimensional compact Hausdorff space. 
Then 
$X$ has no isolated points if and only if 
$(\yomaps{X}{R}, \yomaindis)$  is 
$\youfin(R)$-injective for every range set   $R$ 
with  $3\le \card(R)$. 
\end{cor}

According to 
Proposition \ref{prop:mapsweight} and 
Theorem \ref{thm:stand}, the next proposition 
holds true. Recall that $\yocantorc$ stands for the 
Cantor set. 
\begin{prop}\label{prop:sepUryI}
If $R$ is a countable or finite range set, 
then the space 
$(\yomaps{\yocantorc}{R}, \yomaindis)$ is 
the (separable) $R$-Urysohn universal 
$R$-valued ultrametric space. 
\end{prop}

We denote by $\yoopsp$ the 
one-point compactification of the countable discrete space $\zz_{\ge 0}$. 
Notice that $\yoopsp=\zz_{\ge 0}\sqcup \{\infty\}$. 
As noted in 
Section \ref{sec:intro}, 
the construction of 
$\yomaps{X}{R}$ can be considered as a
generalization of 
Vestfrid's universal space $V$ in 
\cite{MR1354831} (originally, it is denoted by 
``$**$''), 
which is the space of all decreasing sequence $\{x_{i}\}_{i\in \zz_{\ge0}}$ in $R$ such that 
$x_{i+1}\le x_{i}$ for all $i\in \zz_{\ge 0}$ and 
$x_{i}\to 0$ as $i\to \infty$ equipped with 
the metric $u$ defined by 
$u(\{x_{i}\}_{i\in\zz_{\ge 0}}, \{y_{i}\}_{i\in \zz_{\ge 0}})
=\max_{k\in \zz_{\ge 0}}\{x_{k}, y_{k}\}$ if there exists $i\in \zz_{\ge 0}$ with 
$x_{i}\neq y_{i}$; otherwise, the distance  is  $0$. 
By  Proposition \ref{prop:supsup}, 
the space $(V, u)$ can be regarded as 
the space of 
all continuous maps $f\colon \yoopsp\to R$ such that 
$f(\infty)=0$ and $f(n+1)\le f(n)$ and 
$f(n)\to 0$ as $n\to \infty$
equipped with $\yomaindis$. 
Using Theorem  \ref{thm:favoid}, 
we can   prove the injectivity of   $V$. 

\begin{prop}\label{prop:decrease}
The ultrametric space $(V, u)$ explained above
 is   $\youfin(R)$-injective and complete. 
\end{prop}

\section{Spaces of ultrametrics}\label{sec:ultrametrics}

We begin 
with 
a variant of Theorem \ref{thm:favoid} for ultrametrics. 
A point in a topological space 
is
an  \emph{accumulation point} if 
it is not isolated. 
\begin{lem}\label{lem:sugoi}
Let $X$ be a $0$-dimensional 
compact Hausdorff space, 
$R$ be a range set, 
and $E$ be a subset of $\yocmet{X}{R}$. 
Then  $E$ is $\youfin(R)$-injective if and only if 
$E$ satisfies that  for 
every  finite subset $A$ of $E$, 
$\zeta\in A$, 
and 
$r\in R\setminus \{0\}$, 
there exists $g\in E$ such that 
\begin{enumerate}[label=\textup{(B\arabic*)}]
\item\label{item:magmag}
if $h\in A$ satisfies 
$\umetdis_{X}^{R}(\zeta, h)\le r$, then 
there exists $a_{h}\in X$ such that 
$\yosph(a_{h}, r; h)\neq \yosph(a_{h}, r; g)$; 

\item\label{item:mm2} 
we have 
$\umetdis_{X}^{R}(\zeta, g)\le r$. 
\end{enumerate}
\end{lem}
\begin{proof}
First assume that $E$ is $\youfin(R)$-injective. 
Then, owing to  Theorem \ref{thm:favoid}, 
we notice that $E$ is full-attachable. 
The conditions  \ref{item:www} and 
 \ref{item:www2} imply 
 \ref{item:magmag} and 
 \ref{item:mm2}, respectively. 

Next assume that $E$ satisfies 
\ref{item:magmag} and 
 \ref{item:mm2}. 
According to  Theorem \ref{thm:favoid}, 
it suffices to show that 
$E$ is full-attachable in 
$\yomaps{X^{2}}{R}$. 
Take 
a  finite subset $A$ of $E$, 
$\zeta\in A$, 
and 
$r\in R\setminus \{0\}$. 
By \ref{item:magmag}, 
we may assume that  there exists  
$b\in \yosph(a_{h}, r; h)$ and 
$b\not\in \yosph(a_{h}, r, g)$
for all $h\in A$. 
Then 
$(a_{h}, b)\in h^{-1}(r)$ and 
$(a_{h}, b)\not\in g^{-1}(r)$, i.e., 
$h^{-1}(r)\neq g^{-1}(r)$. 
Thus $E$ satisfies the condition \ref{item:www}. 
The condition \ref{item:www2} follows 
from \ref{item:mm2} and the fact that 
$\umetdis_{X}^{R}=\yomaindis|_{\yocmet{X}{R}}$
(see Proposition \ref{prop:spectral}). 
Therefore the set $E$ is full-attachable. 
\end{proof}

\begin{thm}\label{thm:contimetuniv}
Let 
 $X$ be an infinite  $0$-dimensional 
compact Hausdorff space 
possessing  an accumulation point, 
and 
 $R$ be 
a range set. 
Then 
the space 
$(\yocmet{X}{R}, \umetdis_{X}^{R})$ is 
$\youfin(R)$-injective and 
complete. 
\end{thm}

\begin{proof}
Take a finite subset  $A$ of 
$\yocmet{X}{R}$,  
$\zeta\in A$,
and  $r\in R\setminus \{0\}$. 
Let $p$ be an accumulation point of $X$. 
By the compactness of $X$, 
we can take $\{p_{i}\}_{i\in I}$, where 
$I=\{0, \dots, n\}$, 
for which 
the family $\{B(p_{i}, \epsilon)\}_{i\in I}$
is a mutually disjoint covering of $X$. 
Due to Lemma \ref{lem:ultraopcl}, 
we may assume that $p_{0}=p$. 
Put $G=\bigcup_{h\in A} \yosph(p, r; h)$. 
Since 
the set $\yosph(p, r; h)$ 
is clopen in  
$X$ for all  $h\in A$, 
the set 
$G$ is clopen in $X$ and 
$p\not\in  G$. 
Using the  $0$-dimensionality of $X$, 
we can find a 
 clopen neighborhood   $N$ of $p$ such that 
$N\cap G=\emptyset$. 
If necessary, taking  a  sufficiently small neighborhood $N$, 
we may assume that $G$ is a proper subset of 
$X\setminus N$, which is possible since $p$ is an accumulation point. 
We define a pseudo-ultrametric $u_{0}$ on $B(p, r;\zeta)$ by 
\[
u_{0}(x, y)=
\begin{cases}
0 & \text{if $x, y\in N$ or $x, y\in X\setminus N$;}\\
r & \text{otherwise.}
\end{cases}
\]
For each $i\in I\setminus \{0\}$, 
we define a pseudo-ultrametric  $u_{i}$ on 
$B(p_{i}, r)$ by $u_{i}=\zeta$. 
We define a metric $v$ on $I$ by 
$v(i, j)=\zeta(p_{i}, p_{j})$. 
Then 
$\yoquin
=(X, I, v, \{B(p_{i}, r)\}_{i\in I}, \{p_{i}\}_{i\in I})$
 is 
 an $R$-amalgamation system. 
 According to \ref{item:amal:2} in 
 Proposition \ref{prop:amalgam}, 
 the ultrametric $g$ associated with $\yoquin$ satisfies 
 $\umetdis_{X}^{R}(\zeta, g)\le r$. 
 Since $\yosph(p, r; g)=X\setminus N$, 
 we have $\yosph(p, r; g)\neq \yosph(p, r; h)$
 for all $h\in A$. 
 Thus the conditions \ref{item:magmag} and \ref{item:mm2} are satisfied. 
Therefore Lemma \ref{lem:sugoi} that 
$(\yocmet{X}{R}, \umetdis_{X}^{R})$ is 
$\youfin(R)$-injective. 
\end{proof}

Lemma \ref{prop:metricweight} yields the 
next proposition:
\begin{prop}\label{prop:sepUryII}
If $X$ is an infinite  compact ultrametrizable space 
and 
$R$ is a finite or  countable range set, 
then the space 
$(\yocmet{X}{R}, \umetdis_{X}^{R})$ is 
the  (separable) $R$-Urysohn universal 
$R$-valued ultrametric space. 
\end{prop}

Due to the author's results on 
dense subsets of spaces of ultrametrics, 
we obtain the following theorem. 
Recall that $\yocantorc$ is  the 
Cantor set. 

\begin{thm}\label{thm:manyinj}
Let $R$ be a characteristic range set. 
Then 
the following statements are true:
\begin{enumerate}[label=\textup{(\arabic*)}]
\item\label{item:met}
If  $X$ is  an infinite  $0$-dimensional 
 compact metrizable space, 
then $(\ult{X}{S}, \umetdis_{X}^{R})$ is 
$\youfin(S)$-injective. 

\item\label{item:doubling}
If $X$ is  an infinite 
compact ultrametrizable space, 
then 
the set of all doubling ultrametrics in 
$\ult{X}{R}$ is 
$\youfin(R)$-injective. 

\item\label{item:non-doubling}
If $X$ is 
 an infinite
compact ultrametrizable space and 
$R$ is quasi-complete, 
then 
the set of all non-doubling ultrametrics in the space
$\ult{X}{R}$ is 
$\youfin(R)$-injective. 

\item\label{item:up}
If 
 $R$ is exponential, 
then the set of all uniformly perfect ultrametrics in 
$\ult{\yocantorc}{R}$ is 
$\youfin(R)$-injective. 

\item\label{item:nup}
The set of all non-uniformly perfect ultrametrics in 
$\ult{\yocantorc}{R}$ is $\youfin(R)$-injective.

\item\label{item:dimset}
If  $X$ is  an uncountable 
compact ultrametrizable space, 
$\yorbd{a}\in \yorlset$, 
and $R$ is 
$(\yocantorc, \yorbd{a})$-admissible, 
then 
$\yordimsetu{X}{\yorbd{a}}{S}$ is 
$\youfin(R)$-injective. 

\end{enumerate}
\end{thm}
\begin{proof}
According to 
Lemma \ref{lem:denseuniv}, 
it suffices to show 
all sets in the statements 
are dense in $\yocmet{X}{R}$. 
Remark that if $X$ is an infinte compact metrizable space, 
then it has an accumulation point. 
The statement 
\ref{item:met} is deduced 
 from 
Proposition 
\ref{prop:yocsetdense}, 
and 
Theorem \ref{thm:contimetuniv}. 
Using the statement \ref{item:met}, 
in what follows, 
we only need to  show the sets in the statements are dense in 
$\ult{X}{R}$. 
Therefore, 
The statements
\ref{item:doubling}, \ref{item:non-doubling}
follows from 
Proposition 
\ref{thm:doublingdense} and 
Theorem \ref{thm:non-ddense}, respectively. 
The statements  \ref{item:up} and \ref{item:nup} are proven by 
 Theorem  \ref{thm:updense}. 
The statement 
\ref{item:dimset} is guaranteed by 
Theorem \ref{thm:densefrac}. 
\end{proof}

\begin{rmk}
Even if  $X$ is countably infinite, 
with some suitable modifications, 
Theorem  \ref{thm:manyinj} and 
the statement \ref{item:dimset} in 
Theorem
\ref{thm:densefrac} 
are still
 true. 
This is based on the existence 
of a metric in 
$\yordimsetu{\yoopsp}{\rr_{\ge 0}}{\yorbd{a}(u, v)}$, 
for all $u, v\in [0, \infty]$ with $u\le v$,
where 
$\yorbd{a}(u, v)=(0, 0, u, v)$ and $\yoopsp$ is the one-point compactification of 
the countable discrete space. 
Since  details of the exsitence are slightly complicated, 
we omit the proof. 
\end{rmk}

\section{Additional remarks}\label{sec:add}
\label{subsec:topshape}

For an infinite cardinal $\kappa$, 
let $\yonbsp{\kappa}$ denote the 
countable product of the discrete space $\kappa$, which is sometimes called the 
\emph{$0$-dimensional Baire space of weight 
$\kappa$}. 
Notice that $\yonbsp{\aleph_{0}}$ is homeomorphic to 
the space of irrational numbers. 
The following is a topological characterization of 
$\yonbsp{\kappa}$
(see
\cite[Theorem 1]{MR152457}). 

\begin{thm}\label{thm:Bspaces}
Let $\kappa$ be an infinite cardinal and 
 $X$ be a completely ultrametrizable space. 
Assume that
\begin{enumerate}[label=\textup{(\arabic*)}]
\item\label{item:ch:w} 
the space $X$ has a dense subset of cardinal $\kappa$; 
\item\label{item:ch:op} 
 every non-empty open subset of $X$
 contains a closed discrete space of cardinal $\kappa$. 
\end{enumerate}
Then the space $X$ is homeomorphic to 
the
 $0$-dimensional
 Baire  space 
$\yonbsp{\kappa}$ of 
weight $\kappa$. 
\end{thm}

\begin{rmk}
A topological space $X$ satisfies the 
 conditions \ref{item:ch:w} and 
\ref{item:ch:op} in Theorem \ref{thm:Bspaces}
if and only if every  non-empty open subset $O$ of $X$
satisfies $\yoweight{O}=\kappa$. 
\end{rmk}

In 
\cite{Koshino1, Koshino2}, 
Koshino determined the 
topological types of the spaces
$\met(X)$ of metrics on 
a metrizable space $X$ generating the 
same topology of $X$ equipped with the uniform topology or the compact-open topology. 
We now  show a non-Archimedean analogue of 
Koshino's work. 
\begin{thm}\label{thm:topshape}
Let $\kappa$ be an infinite cardinal,  
$X$ be a $0$-dimensional compact Hausdorff space 
with $\yoweight{X}=\kappa$ possessing no isolated points, 
and 
$Y$ be  a $0$-dimensional compact Hausdorff space 
with $\yoweight{X}=\kappa$ possessing an 
accumulation point. 
Let  $R$ be a characteristic range set.
Put $\tau=\max\{\kappa, \card(R)\}$. 
If either of 
$\card(R)\le \kappa$ or 
$\card(R\cap[0, r])=\card(R)$ for all 
$r\in R\setminus \{0\}$,  
then $\yomaps{X}{R}$ and 
$\yocmet{X}{R}$ are 
homeomorphic to $\yonbsp{\tau}$. 
\end{thm}
\begin{proof}
Based on Theorem \ref{thm:Bspaces}, 
the case of $\yomaps{X}{R}$ is deduced from 
Propositions 
\ref{prop:mapsweight} and 
\ref{prop:funcballweight}.
Similarly, the case of 
$\yocmet{X}{R}$
 follows from 
Propositions \ref{prop:metricweight} and
\ref{prop:metballweight}. 
\end{proof}

\begin{cor}\label{cor:topshsh}
For every countable characteristic  range set $R$, 
the $R$-Urysohn universal ultrametric spaces is 
homeomorphic to the space of irrational numbers. 
\end{cor}

\begin{rmk}
Theorem \ref{thm:topshape} is a
generalization of the argument in  the proof of 
\cite[Theorem 2]{MR1843595}. 
\end{rmk}

\begin{rmk}
Let $\kappa$, $R$, $X$ and $Y$ be the same objects in 
Theorem \ref{thm:topshape}. 
Of cause, if $\kappa=\aleph_{0}$ and $R$ is finite or countable, then $(\yomaps{X}{R}, \yomaindis)$ and 
$(\yocmet{X}{R}, \umetdis_{X}^{R})$ are isometric to each other (see \cite{MR2754373}). 
However, in general, 
the author does not know  they 
are isometric. 
In  \cite{Ishiki2023Ury}, 
the author proves  that if $\kappa=\aleph_{0}$ and 
$R$ is uncountable, the two spaces explained above and 
the non-Archimedean Gromov-Hausdorff associated with 
$R$ are isometric to each other. 
\end{rmk}
\begin{rmk}
For every infinite cardinal $\kappa$, 
there exists a $0$-dimensional compact Hausdorff space  $X$
with $\yoweight{X}=\kappa$ possessing  no isolated points. 
For example, the product 
$\{0, 1\}^{\kappa}$ is such a space. 
It also becomes an example of 
a $0$-dimensional compact Hausdorff space  $X$
with $\yoweight{X}=\kappa$ possessing  an accumulation point. 
\end{rmk}

\begin{ques}\label{ques:more}
For every  range set $R$, 
are there more  constructions of $\youfin(R)$-injective 
$R$-ultrametric spaces?
\end{ques}

\begin{ac}
The author would like to thank 
Tatsuya Goto for helpful comments. 
\end{ac}

\bibliographystyle{amsplain}
\bibliography{../../../bibtex/UU.bib}

\end{document}